\documentclass[11pt,leqno]{article}
\usepackage{amsthm,amsfonts,amssymb,amsmath,oldgerm}
\usepackage{epsfig}
\numberwithin{equation}{section} 

\newcommand{\btheta}{{\bar \theta}}

\renewcommand\d{\partial}

\def\Trace{\rm Trace }

\def\eps{\varepsilon }

\newcommand{\E}{{\mathcal {E}}}
\newcommand{\M}{{\mathcal {M}}}

\renewcommand\d{\partial}

\newcommand\R{\mathbb R}

\def\eps{\varepsilon}

\newcommand\br{\begin{remark}}
\newcommand\er{\end{remark}}
\newcommand\bp{\begin{pmatrix}}
\newcommand\ep{\end{pmatrix}}
\newcommand\be{\begin{equation}}
\newcommand\ee{\end{equation}}
\newcommand\ba{\begin{equation}\begin{aligned}}
\newcommand\ea{\end{aligned}\end{equation}}

\newcommand{\bap}{\begin{app}}
\newcommand{\eap}{\end{app}}
\newcommand{\begs}{\begin{exams}}
\newcommand{\eegs}{\end{exams}}
\newcommand{\beg}{\begin{example}}
\newcommand{\eeg}{\end{exaplem}}
\newcommand{\bpr}{\begin{proposition}}
\newcommand{\epr}{\end{proposition}}
\newcommand{\bt}{\begin{theorem}}
\newcommand{\et}{\end{theorem}}
\newcommand{\bc}{\begin{corollary}}
\newcommand{\ec}{\end{corollary}}
\newcommand{\bl}{\begin{lemma}}
\newcommand{\el}{\end{lemma}}
\newcommand{\bd}{\begin{definition}}
\newcommand{\ed}{\end{definition}}
\newcommand{\brs}{\begin{remarks}}
\newcommand{\ers}{\end{remarks}}

\newtheorem{theo}{Theorem}[section]

\newtheorem{exams}[theo]{Examples}

\numberwithin{equation}{section}

\newcommand{\CalT}{\mathcal{T}}

\newcommand{\RR}{{\mathbb R}}

\newcommand{\CC}{{\mathbb C}}

\newcommand{\Span}{{\rm Span }}

\newtheorem{theorem}{Theorem}[section]
\newtheorem{proposition}[theorem]{Proposition}
\newtheorem{corollary}[theorem]{Corollary}
\newtheorem{lemma}[theorem]{Lemma}
\newtheorem{definition}[theorem]{Definition}

\newtheorem{example}[theorem]{Example}
\newtheorem{remark}[theorem]{Remark}

\title{
High-frequency asymptotics and 1-D stability of ZND detonations
in the small-heat release and high-overdrive limits}

\author{\sc \small 
Kevin Zumbrun\thanks{Indiana University, Bloomington, IN 47405;
kzumbrun@indiana.edu:
Research of K.Z. was partially supported
under NSF grants no. DMS-0300487 and DMS-0801745.
 }}
\begin{document}

\maketitle


\begin{abstract}
We establish one-dimensional spectral, or ``normal modes'', 
stability of ZND detonations in the small heat release limit 
and the related high overdrive limit with heat release and activation 
energy held fixed, verifying numerical observations of Erpenbeck in the 1960s.
The key technical points are a strategic rescaling of parameters 
converting the infinite overdrive limit to a finite, regular 
perturbation problem, and a careful high-frequency analysis 
depending uniformly on model parameters.
The latter recovers the important result of high-frequency 
stability established by Erpenbeck by somewhat different techniques.
Notably, the techniques used here yield quantitative estimates 
well suited for numerical stability investigation.
\end{abstract}


\section{Introduction}
In this note, we establish one-dimensional
spectral stability in the small heat release
and high overdrive limits of ZND detonations, confirming 
numerical observations of Erpenbeck going back to \cite{Er2},
but up to now not rigorously verified.
In the process, we recover by a somewhat different argument
the fundamental result of Erpenbeck
\cite{Er3} that detonations are one-dimensionally stable with
respect to sufficiently high frequencies.

The basic argument for stability in the high overdrive limit,
based on a strategic rescaling of parameters
converting the problem to a small-heat release/small activation
energy/maximal shock strength limit on a bounded parameter
range, was indicated already in \cite{Z1}.
However, the result seems sufficiently fundamental to warrant
an exposition giving full detail.
In particular, the discussion of \cite{Z1} concerned only bounded
frequencies, for which stability follows by a simple continuity argument.
Stability for high frequencies can be concluded
from a well-known result of \cite{Er3}, which, restricted to
the one-dimensional setting, asserts that
instabilities cannot occur outside a sufficiently large ball.
However, the arguments of \cite{Er3}, based on 
semiclassical limit and turning point theory,\footnote{
Specifically, validation of a formal asymptotic expansion in one over the 
frequency for complexified $x$.}
do not readily yield quantitative estimates on rates of convergence
or dependence on parameters.
As continuous dependence on parameters of the radius outside which
instabilities are excluded is crucial for the limiting argument
described above, it seems useful to revisit the high-frequency
limit problem in greater detail.

Moreover, as pointed out in \cite{CJLW}, the 
issue of matching at $x\to -\infty$ of the formal asymptotic
solution with the solution prescribed by the required behavior
at spatial infinity appears to require 
a bit more discussion beyond what
is given in \cite{Er3}, where it is 
concluded simply from the observation that the limits
as $x$ or frequency go to infinity commute.
This argument seems to require either uniform convergence
on all of $x\in (-\infty,0]$ of the formal asymptotic series as frequency
goes to infinity, which (to us) does not appear obvious, or
else uniform estimates independent of frequency
on behavior as $x\to -\infty$.

These aspects (uniform dependence and uniform convergence on $(-\infty,0]$)
of high-frequency behavior
are the main issues addressed here, 
where they are treated
by 
a careful application of the asymptotic ODE
techniques developed in \cite{GZ,MeZ1,MaZ3,PZ,Z1}.
These in turn are natural
outgrowths of the classical asymptotic ODE techniques developed by Coddington, 
Levinson, Coppel, and others, as described in \cite{CL,Co} 
and references therein, including those cited by 
Erpenbeck \cite{Er1,Er2} in describing behavior as $x\to -\infty$.\footnote{
See problem 29, p. 104 of \cite{CL}, cited in \cite{Er1}.}
Our arguments are qualitatively different from the ones of Erpenbeck
based on semiclassical limit/turning point 
theory,
making use at a key point of the exponential convergence of profiles to
a limit as $x\to -\infty$.
Specifically, in the neutral case that diagonal elements have 
uniformly small
spectral gap, 
we apply a variable-coefficient version
(Lemma \ref{varconlem})
of the conjugation lemma of \cite{MeZ1}
to close the argument, extending and
refining the related constant-coefficient gap lemma estimates 
of Proposition 5.7, \cite{CJLW}, applying to a single mode.\footnote{
These estimates, valid on $(-\infty, -C\log |\lambda|]$ for frequency 
$|\lambda|\to \infty$, do not seem sufficient for our purposes.}

Notably, the simple and concrete estimates thus derived 
yield quantitative bounds of potential use for numerical
stability investigations.
Recall \cite{HuZ1} that computational intensity
of numerical stability computations increases 
rapidly with increasing frequency, so that bounds on frequency
are of considerable practical interest for applications.
See \cite{BZ} for a first effort in this direction in
the simplified context of Majda's model \cite{M}.

\section{Equations and assumptions} \label{s:ideal}
The reacting Euler, or Zeldovitch--von Neumman--Doering (ZND)
equations commonly used to model combustion, 
expressed in Lagrangian coordinates, are
\begin{equation}\label{ZND}
\left\{ \begin{aligned}
 \d_t \tau - \d_x u & = 0,\\
 \d_t u + \d_x p  & = 0,\\
 \d_t E + \d_x(pu) & = qk 
\phi(T) 
z ,\\
 \d_t z & = - k 
\phi(T) 
z ,\\
\end{aligned}\right.
\end{equation}
where $\tau>0$ denotes specific volume, $u$ velocity, 
$ E = e+ \frac{1}{2} u^2$ 
specific gas-dynamical energy,
$e>0$ specific internal energy,  
and $0 \leq z \leq 1$ mass fraction of the reactant.
Here, $k > 0$ measures reaction rate and $q$ heat release of the reaction,
with $q>0$ corresponding to an exothermic reaction and $q<0$ to
an endothermic reaction,  
while $T=T(\tau,e,z)>0$ represents temperature and $p=p(\tau,e,z)$ pressure.

The equations are of quasilinear hyperbolic type provided that
(but not only when) 
\be\label{pos}
(p,T)= (p,T)(\tau,e) \quad
\hbox{\rm and }\; p, p_\tau, T, T_e>0.
\ee
For simplicity, we assume throughout this paper
an ideal gas equation of state and Arrhenius-type ignition function,
\be\label{eos}
 p=\Gamma \tau^{-1} e, \quad T=c^{-1} e, 
\quad \phi(T)=e^{-\frac{\mathcal{E}}{T}}
\ee
where $E=e+u^2/2$ is specific (gas-dynamical) energy, 
$c>0$ is the specific heat constant,
$\Gamma>0$ is the Gruneisen constant, 
and $\mathcal{E}\ge 0$ is activation energy.
Our results on small heat-release and high-frequency stability
clearly extend to the general case \eqref{pos}; however, our
main results, on the high-overdrive limit, depend in an essential
way on the invariances associated with \eqref{eos}.

\section{Detonation profiles and parametrization}
A {\it right-going strong detonation wave} is a 
traveling-wave solution
\be\label{profile}
(u,z)(x,t)=(\bar u, \bar z)(x-st), 
\quad
\lim_{x\to  -\infty} (\bar u,\bar z)(x)=(u_-, 0),
\quad
(\bar u,\bar z)(x)\equiv (u_+, 1) \; \hbox{\rm for}\; x\ge 0
\ee
of \eqref{ZND} with speed $s>0$
connecting a burned state on the left to an unburned state on the right,
smooth for $x\le 0$, with a Lax-type gas-dynamical shock at $x=0$,
known as the {\it Neumann shock}.

Rescaling
$$
\begin{aligned}
(x,t,s, \tau,u,T) &\to 
\Big(\frac{\tau_+ sx}{L}, \frac{\tau_+ s^2t}{L}, 1,
\frac{\tau}{\tau_+}, \frac{u}{\tau_+  s}, 
\frac{T}{\tau_+^2 s^2}\Big), 
\quad
(z,q,k,\mathcal{E})\to 
\Big( z, \frac{q}{\tau_+^2 s^2}, \frac{Lk}{\tau_+s^2} ,
\frac{\mathcal{E}}{\tau_+^2 s^2}\Big) \\
\end{aligned}
$$
following \cite{Z1}, 
we may take without loss of generality $s=1$, $\tau_+=1$
and (by translation invariance in $u$), $u_+=0$,
leaving $e_+>0$ as the parameter determining the Neumann
shock.

By explicit computation (\cite{Z1}, Appendix C), we have then 
\be\label{idealzndprof}
\bar u= 1-\bar \tau,
\quad
\bar e= 
\frac{\bar \tau(\Gamma e_++1-\bar \tau)}{\Gamma},
\ee
\ba
\bar \tau&= 
\frac{
(\Gamma+1) (\Gamma e_++1)-
\sqrt{ 
(\Gamma+1)^2 (\Gamma e_++1)^2
- \Gamma (\Gamma +2) ( 1+2(\Gamma +1)e_+ -2q(\bar z-1))  }
}
{\Gamma +2} ,\\
\ea
where
\ba\label{cjlim}
0\le q\le q_{cj} & :=
\frac{ (\Gamma+1)^2(\Gamma e_+ + 1)^2- \Gamma (\Gamma+2) ( 1+2(\Gamma +1)e_+ ) }
{2\Gamma(\Gamma+2)},
\ea
and
$\bar z'=k\phi(c^{-1}\bar e(\bar z))\bar z$;
in the simplest case $\mathcal{E}=0$,
$\bar z=e^{kx}$.

The jump at the preceding ``Neumann shock'' at $x=0$ is given
(see \cite{Z1}, App. C) by
\be\label{jump}
[\bar W]:=
\Big( 
1-\bar \tau(0), \bar \tau(0)-1,  e_+-\bar e(0), 0 \Big)^T.
\ee

Taking finally $k=1$ by a simultaneous rescaling of $x$ and $t$
if necessary, we can parametrize all possible ZND profiles by 
\be\label{param}
(e_+, q, \mathcal{E}, \Gamma), 
\ee
where 
$0\le e_+ \le \frac{1}{\Gamma(\Gamma+1)}$,
$0\le q\le q_{cj}(e_+)$,
$0\le \mathcal{E} < \infty$, and $0<\Gamma<\infty$.
See \cite{Z1} for further details.

\subsection{The high-overdrive limit and the scaling of Erpenbeck}\label{s:f}

A similar scaling was used by Erpenbeck in
\cite{Er3}, but with $e_+$ held fixed instead of wave speed $s$.
Converting from Erpenbeck's to our scaling amounts to rescaling the
wave speed, so that $T\to T/s^2$ and $\mathcal{E}\to 
\mathcal{E}/s^2$, and $t\to ts^2$
($u$ is translation invariant, so irrelevant).
Thus, as noted in \cite{Z1},
the high-overdrive limit discussed in \cite{Er3}, in which
$s\to \infty$ with $u_+$ held fixed, corresponds in our scaling
to taking $\mathcal{E}= \mathcal{E}_0 e_+$, 
$q= q_0 e_+$, and varying $e_+$ from $e_+=e_{cj}(q_0) $
 ($s$ minimum) to $ 0$ ($s=\infty$), where $e_{cj} $
is determined implicitly by the relation $q_{cj}(e_{cj})=q_0 e_{cj}$:
that is, the simultaneous {\it zero heat release},
{\it zero activation energy}, and {\it strong shock limit}
$e_+\to 0$, 
$\mathcal{E}\to 0$,
and $q\to0$.
 
\section{Asymptotic ODE theory}\label{asymptotic}

\subsection{The conjugation lemma}\label{s:conj}
Consider a general first-order system
$W'=A^p(x,\lambda)W$, $W\in \CC^N$, $\lambda\in \CC$, $x\in \RR$,
with asymptotic limit $A^p_-$ as $x\to - \infty$,
where $p\in \RR^m$ denote model parameters and $'$ $d/dx$.

\bl [\cite{MeZ1,PZ}]\label{conjlem}
Suppose for fixed $\theta>0$ and $C>0$ that 
\be\label{udecay}
|A^p-A^p_-|(x,\lambda)\le Ce^{-\theta |x|}
\ee
for $x\le 0$ uniformly for $(\lambda,p)$ in a neighborhood of 
$(\lambda_0, p_0)$ and that $A$ varies analytically in $\lambda$ 
and continuously in $p$ as a function into $L^\infty(x)$.
Then, there exists in a neighborhood of $(\lambda_0,p_0)$ 
an invertible linear transformation
$P^p(x,\lambda) =I+\Psi^p(x,\lambda)$ defined
on $x\le 0$, analytic in $\lambda$ and continuous in $p$ 
as a function into $L^\infty [0,\pm\infty)$, such that
\begin{equation}
\label{Pdecay} 
| \Psi^p |\le C_1 e^{-\bar \theta |x|}
\quad
\text{\rm for } x\le 0,
\end{equation}
for any $0<\btheta<\theta$, some $C_1=C_1(\bar \theta, \theta)>0$,
and the change of coordinates $W=:P^p Z$ reduces 
$W'=A^pW$
to the constant-coefficient limiting system
$Z'=A^p_- Z $
for  $x\le 0$.
\el

\begin{proof} See the proof of Lemma 2.5, \cite{Z4}, or Lemma A.1, \cite{Z1}.
\end{proof}

\subsection{The tracking lemma}\label{s:track}

Consider an approximately block-diagonal system
\begin{equation}
W'= \bp M_1 & 0 \\ 0 & M_2 \ep(x,p) W + \delta(x,p) \Theta(x,p) W,
\label{blockdiag}
\end{equation}
where $\Theta$ is a uniformly bounded matrix, $\delta(x)$ scalar, 
and $p$ a vector of parameters,
satisfying a pointwise spectral gap condition
\begin{equation} \label{gap}
\min \sigma(\Re M_1^p)- \max \sigma(\Re M_2^p)
\ge \eta(x) >0
\, \text{\rm for all } x.
\end{equation}
(Here as usual $\Re N:= \frac{1}{2}(N+N^*)$ denotes the
``real'', or symmetric part of $N$.)

\begin{lemma}[\cite{MaZ3,PZ,Z1}] \label{reduction}
Consider a system \eqref{blockdiag} under the gap assumption
\eqref{gap}, with $\Theta^p$ uniformly bounded and
$\eta\in L^1_{\rm loc}$.
If $\sup (\delta/\eta)(x)$ is sufficiently small,
then there exists a unique linear
transformation $\Phi(x,p)$,
possessing the same regularity with respect to $p$ 
as do coefficients $M_j$ and $\delta\Theta$ 
(as functions into $L^\infty(x)$), 
for which the graph $\{(Z_1, \Phi Z_1)\}$ is invariant under 
\eqref{blockdiag}, and 
\ba\label{ptwise}
|\Phi^p(x)| \le C
\int_{-\infty}^{x} e^{\int_y^x -\eta(z)dz} \delta(y)dy
\le \sup_{(-\infty,x]} (\delta/\eta).
\ea
\end{lemma}

\begin{proof}
See the proof of Lemma A.4 together with Remark A.6 in \cite{Z1}.
\end{proof}

\subsection{A variable-coefficient conjugation lemma}\label{s:varcon} 

The key new technical contribution of this paper at the level of
asymptotic ODE is the following 
simple observation.
Consider a first-order system
\be\label{newsys}
W'=A^p(x;\eps)W:=M^p(x;\eps)W +\Theta^p(x;\eps)W,\quad x\le 0,
\ee
$W\in \CC^N$,  $x\in \RR$, $p\in \RR^m$,
with distinguished parameter $\eps\to 0$, satisfying
\be\label{udecay2}
|\Theta^p(x,\eps)|\le C \eps^2 e^{-\theta \eps |x|}, 
\ee
\be\label{neutral}
|\Re M^p(x,\eps)|\le \eps \delta^p(\eps) + C \eps e^{-\theta \eps |x|}
\ee
for some uniform $C,\theta>0$, all $x\le 0$, 
where $\Re M:=\frac{1}{2}(M+M^*)$.

\bl \label{varconlem}
For $\delta^p(\eps)\le \delta_*$ sufficiently small, and $\eps>0$ 
sufficiently small, there exists an invertible linear transformation
$P^p(x,\eps) =I+\Psi^p(x,\eps)$ defined on $x\le 0$
such that
\begin{equation}
\label{Pdecay2} 
| \Psi^p |\le C_1 \eps e^{- \theta \eps |x|/2}
\quad
\text{\rm for } x\le 0,
\end{equation}
and the change of coordinates $W=:P^p Z$ reduces 
$W'=A^pW$ to $Z'=M^p Z $. 
\el

\begin{proof}
Equivalently, we construct a solution $P^p$ of the (matrix-valued)
homological equation
\be\label{homolog}
P'=
\mathcal{M}^pP + \Theta^p P,
\qquad
\mathcal{M}P:=M^pP-PM^p,
\ee
satisfying $P^p\to I$ as $x\to -\infty$, or, equivalently, a 
solution $\Psi^p$ of the integral fixed-point equation
\begin{equation}
\begin{aligned}
\CalT \Psi(x) 
&=  \int^x_{-\infty} \mathcal{F}^{y\to x} \Theta (y)(I+ \Psi(y)) dy ,
\end{aligned}
\end{equation}
where $\mathcal{F}^{y\to x}$ is the solution operator of
$P'=\mathcal{M}^p P$ from $y$ to $x$.

Denoting by $(P:Q):=\Trace (P^*Q)$ the Frobenius inner product,
and $\|P\|:=(P,P)^{1/2}$ the Frobenius matrix norm, we find that
\ba\label{enest}
\frac{1}{2}\|P\|^2=\Re (P:P')= \Re (P:M^pP-PM^p)
&= (P:(\Re M^p)P-P(\Re M^p))\\
&\le 2 \|\Re M^p\|\|P\|^2,
\ea
yielding by \eqref{neutral} the bound
\be\label{growth}
\|\mathcal{F}^{y\to x}\|\le Ce^{ 2 \eps \delta^p(\eps) (x-y)}
\le Ce^{ 2 \eps \delta_* (x-y)}.
\ee 

For $ \delta_*\le \theta/4$ and $\eps>0$ sufficiently small, this implies that
$\mathcal{T}$ is a contraction on $L^\infty(-\infty, 0]$.
For, applying \eqref{udecay2}, we have
\begin{equation}\label{con}
\begin{aligned}
\left|\CalT \Psi_1 - \CalT \Psi_2 \right|_{(x)} 
&\le C\eps^2 |\Psi_1 - \Psi_2|_\infty 
\int^x_{-\infty} e^{\theta \eps (x-y)/2} e^{\theta \eps y} dy 
\le  C_1 \eps |\Psi_1 - \Psi_2|_\infty e^{-\theta \eps |x|/2},
\end{aligned}
\end{equation}
which for $\eps$ sufficiently small is less than $\frac{1}{2} |\Psi_1 - \Psi_2|_\infty$.

By iteration, we thus obtain a solution 
$\Psi \in L^\infty (-\infty, 0]$ of $\Psi = \CalT \Psi$. 
Further, taking $\Psi_1=\Psi$, $\Psi_2=0$ in \eqref{con}, we obtain,
using contraction together with 
the final inequality in \eqref{con}, that 
$ |\Psi - \CalT(0)|_{L^\infty(-\infty,x)} \le  
\frac{1}{2}|\Psi-0|_{L^\infty(-\infty,x)}, $
 yielding, as claimed,
$
 |\Psi_{L^\infty(-\infty,x)}|\le 2 |\CalT(0)_{L^\infty(-\infty,x)}|
\le 
2C_1 \eps e^{-\theta \eps |x|/2}.
$
\end{proof}

\section{The Evans--Lopatinski determinant}\label{evansZND}

We now briefly recall the linearized stability
theory of \cite{Er1,JLW,Z1,HuZ2}. 
Shifting to coordinates $\tilde x=x-st$ moving with the background
Neumann shock, write \eqref{ZND} as
$W_t + F(W)_x=R(W)$,
where
\begin{equation}\label{abcoefs}
\begin{aligned}
W:=\begin{pmatrix} \tau\\u\\E\\z \end{pmatrix}, \quad
F:=\begin{pmatrix} 
-u-s\tau\\
\Gamma e/\tau -su\\
u\Gamma e/\tau -sE\\
-sz \end{pmatrix},\quad
R:=\begin{pmatrix} 0 \\ 0 \\ qkz\phi(u)\\-kz\phi(u) \end{pmatrix}.\\
\end{aligned}
\end{equation}
To investigate solutions in the vicinity of a discontinuous 
detonation profile, we postulate existence of a single shock
discontinuity at location $X(t)$, and reduce to a fixed-boundary
problem by the change of variables $x\to x-X(t)$.
In these coordinates, the problem becomes
$W_t + (F(W) - X'(t) W)_x=R(W)$, $x\ne 0$,
with jump condition
$X'(t)[W] - [F(W)]=0$, 
$[h(x,t)]:=h(0^+,t)-h(0^-,t)$ as usual denoting jump 
across the discontinuity at $x=0$.

\subsection{Linearization/reduction to homogeneous form}\label{linearization}
In moving coordinates, $\bar W^0$ 
is a standing detonation, hence 
$(\bar W^0, \bar X)=(\bar W^0,0)$ is a steady solution of 
the nonlinear equations.
Linearizing about $(\bar W^0,0)$, we obtain the {\it linearized equations} 
$(W_t - X'(t)(\bar W^0)'(x)) + (AW)_x =EW,$
with jump condition
$X'(t)[\bar W^0] - [A W]=0$ at $x=0$,
where
$A:= (\partial/\partial W)F$, $E:= (\partial/\partial W)R$.
Computing, we have
\be \label{ae}
A= \bp 0 & -1 & 0 & 0\\
-\frac{\Gamma \bar e}{\bar \tau^2} & -\frac{\Gamma \bar u}{\bar \tau} &
\frac{\Gamma }{\bar \tau} & 0\\
-\frac{\bar u\Gamma \bar e}{\bar \tau^2} & \frac{\Gamma(\bar e- \bar u^2)}{\bar \tau} &
\frac{\Gamma \bar u }{\bar \tau} & 0\\
0 & 0 & 0 & 0\\
\ep -sI , \qquad  
E = \begin{pmatrix}0 & 0 & 0 & 0\\0 & 0 & 0 &0\\ 
0 & \frac{q k\, d\phi(\bar T)\bar u\bar z}{c} &\frac{ qk \, d\phi(\bar T)\bar z }{c}& qk\phi(\bar T)\\ 
0 & -\frac{k\, d\phi(\bar T)\bar u \bar z}{c} & -\frac{k\, d \phi(\bar T)\bar z }{c}&-k\phi(\bar T) \end{pmatrix} . 
\ee


Reversing the original transformation to linear order, following \cite{JLW},
by the change of variables $W\to W- X(t)(\bar W^0)'(x)$,
and noting that $x$-differentiation of 
the steady profile equation $F(\bar W^0)_x=R(\bar W^0)$ gives
$(A(\bar W^0)(\bar W^0)'(x))_x=E(\bar W^0)(\bar W^0)'(x)$,
we obtain modified, {\it homogeneous} interior equations
$W_t + (AW)_x =EW$
together with a modified jump condition
accounting for front dynamics of
$X'(t)[\bar W^0]-  [A \big(W+ X(t) (\bar W^0)'\big) ]=0$.

\bl\label{Alem}
For $0\le q< q_{cj}$, $A(x)=dF(\bar W(x))$ is invertible for all $x$,
with 
\be\label{Adecay}
|A(x)-A_-|\le Ce^{-\eta |x|}
\ee
for all $x\le 0$, some $C,\eta>0$, where $A_-:=dF(W_-)$.
\el

\begin{proof} Direct calculation.
(The property $\det A_-=0$ defines $q_{cj}$, marking
the boundary of existence of detonation profiles; see \cite{Z1}.)
\end{proof}

\subsection{The stability determinant}
Seeking normal mode solutions $W(x,t)=e^{\lambda t}W(x)$,
$X(t)=e^{\lambda t}X$,
$W$ bounded, of the linearized homogeneous equations,
we are led to the generalized eigenvalue equations
$(AW)' = (-\lambda I   + E)W$ for $x\ne 0$, and
 $X(\lambda[\bar W^0]-[A (\bar W^0)']) - [A W]=0$, 
where ``$\prime$'' denotes $d/dx$, or, setting $Z:=AW$, to
\be\label{eig}
Z' = GZ, \quad x\ne 0,
\ee
\begin{equation}\label{eigRH}
\begin{aligned}
X(\lambda[\bar W^0]-[A (\bar W^0)']) - [Z]=0, 
\end{aligned}
\end{equation}
with
\begin{equation}\label{G0}
G:=(-\lambda I  + E)A^{-1},
\end{equation}
where we are implicitly using the fact that $A$ is
invertible, i.e., avoiding the limiting, Chapman--Jouget case $q=q_{CJ}$.

\begin{lemma}[\cite{Er1,JLW}] \label{lem:splitting} 
For $q\ne q_{CJ}$,
on $\R \lambda >0$, the limiting $(n+1)\times (n+1)$
coefficient matrices $G_\pm:= \lim_{z\to \pm \infty} G(z)$ 
have unstable subspaces of fixed rank: full rank $n+1$ for
$G_+$ and rank $n$ for $G_-$.
Moreover, these subspaces extend analytically to $\R \lambda \le -\eta< 0$.
\end{lemma}

\begin{proof}
Straightforward calculation using upper-triangular form
 of $G_\pm$ \cite{Er1,Er2,Z1,JLW}. 
\end{proof}

\begin{corollary}[\cite{Z1,JLW}] \label{cor:splitting} 
For $q\ne q_{cj}$, On $\R \lambda >0$, the only bounded solution of \eqref{eig}
for $x>0$ is the trivial solution
 $W\equiv 0$.  For $x<0$, the bounded solutions consist of
an $(n)$-dimensional
manifold $\Span \{Z_1^+, \dots, Z_{n}^+\}(\lambda,x)$
of exponentially decaying solutions, analytic in $\lambda$ and
continuous in parameters $(e_+, q, \mathcal{E}, \Gamma)$, and
tangent as $x\to -\infty$ to the subspace of exponentially decaying
solutions of the limiting, constant-coefficient equations
$Z'=G_-Z$; moreover, 
this manifold extends analytically to $\R \lambda \le -\eta< 0$.
\end{corollary}

\begin{proof}
The first observation is immediate, using the fact that $G$
is constant for $x>0$, with eigenvalues of positive real part.
The second follows from standard asymptotic ODE theory, 
Lemma \ref{conjlem}, Appendix \ref{asymptotic},
together with the fact that $G$, by Lemma \ref{Adecay}, decays exponentially
to its limit $G_-:=G(-\infty)$ as $x\to -\infty$, and that $G_-$ by
direct calculation has a single eigenvalue of negative real part for
$\Re \lambda>0$, which extends analytically to $\Re \lambda=0$ (by
spectral separation from the remaining spectra of $G_-$) and continuously
in $(e_+, q, \mathcal{E}, \Gamma)$.
\end{proof}

\begin{definition}\label{lopdef}
We define the {\it Evans--Lopatinski determinant}
\begin{equation}\label{evanseq}
\begin{aligned}
D_{ZND}(\lambda)&:=
\det \begin{pmatrix}
Z_1^-(\lambda,0), & \cdots, & Z_{n}^-(\lambda,0), &
\lambda [\bar W^0]-[A (\bar W^0)']\\
\end{pmatrix}\\
&=\det \begin{pmatrix}
Z_1^-(\lambda,0), & \cdots, & Z_{n}^-(\lambda,0), &
\lambda[\bar W^0]+R(\bar W^0)(0^-) \end{pmatrix},
\end{aligned}
\end{equation}
where $Z^-_j(\lambda,x)$ are as in Corollary \ref{cor:splitting}.
\end{definition}

The analytic function $D_{ZND}$ is exactly the {\it stability function} 
derived in a different form by Erpenbeck \cite{Er1,Er2}.
Evidently (by \eqref{eigRH} combined with Corollary \ref{cor:splitting}), 
$\lambda$ is a generalized eigenvalue/normal mode
for $\R \lambda \ge 0$ if and only if $D_{ZND}(\lambda)=0$.

By duality, the zeros of $D_{ZND}$ agree
with those of the ajoint formulation
\begin{equation}\label{adevanseq}
\tilde D_{ZND}(\lambda)=
\langle \tilde Z,  \lambda[\bar W^0]+R \rangle|_{x=0},
\end{equation}
where $\langle \cdot, \cdot \rangle$ denotes complex inner product
and $\tilde Z$ denotes an analytically chosen solution of
$\tilde Z' = -G^* \tilde Z$, $x\le 0$
that is decaying as $x\to -\infty$ (see \cite{HuZ1,CJLW,Z1}). 

\begin{definition}\label{stabdef}
For $q\ne q_{cj}$ a ZND detonation is spectrally (or ``normal modes'') 
stable if the only zero of $D_{ZND}$ (equivalently of $\tilde D_{ZND}$)
on $\Re \lambda \ge 0$ is a single 
zero of mulitplicity one at $\lambda=0$ (necessarily
at least multiplicity one by translational invariance).
\end{definition}

\section{Continuous dependence and bounded-frequency stability}

\bpr
For $q\ne q_{cj}$, $D_{ZND}$ and $\tilde D_{ZND}$ are analytic in
$\lambda$ and vary continuously in $(e_+, q, \mathcal{E}, \Gamma)$,
uniformly on compact subsets of $\{\Re \lambda \ge 0\}$.
\epr

\begin{proof} 
Immediate, by the construction of the previous section.  
\end{proof}

\bc\label{qcor}
Under the parametrization \eqref{param},
ZND detonations are stable with respect
to bounded frequencies $\{\Re \lambda \ge 0\}\cap \{|\lambda|\le R\}$,\footnote{
That is, $D_{ZND}$ (resp. $\tilde D_{ZND}$) vanishes on this set only
at a zero of multiplicity one at $\lambda=0$.} 
any $R>0$, in the small heat release limit $q\le q_*$ sufficiently small
for $\Gamma$, $\mathcal{E}$, $e_+$ bounded and some $q_*>0$.
\ec

\begin{proof} 
By continuity with respect to $q$, and the properties of uniform limits
of analytic functions, the zeros of $\tilde D_{ZND}$ on 
$\{\Re \lambda \ge 0\}\cap \{|\lambda|\le R\}$ converge as $q\to 0$
to the zeros of $\tilde D_{ZND}$ with $q=0$.
Noting that the $u$ and $z$ equations decouple for $q=0$, and
the profile $\bar u$ reduces to a gas-dynamical shock,
we find readily from \eqref{adevanseq} that $\tilde D_{ZND}$
reduces to the Lopatinski determinant for a gas-dynamical shock,
which is known (see \cite{Er4,M}) to vanish only at a zero of multiplicity one
at $\lambda=0$. 
\end{proof}

\bc\label{fcor}
Under the parametrization \eqref{param},
ZND detonations are stable with respect
to bounded frequencies $\{\Re \lambda \ge 0\}\cap \{|\lambda|\le R\}$,
any $R>0$, in Erpenbeck's high-overdrive limit, or, in our scaling,
$q,\mathcal{E},e_+\to 0$.
\ec

\begin{proof}
Immediate, by Corollary \ref{qcor} and the description of Section \ref{s:f}.
\end{proof}

\section{High-frequency asymptotics and large-frequency stability}\label{s:HF}

\subsection{Approximate diagonalization}\label{s:diag}
For $q$ bounded away from $q_{cj}$, $k=1$, $\Gamma$ bounded,
suppressing parameters $p=(\Gamma, \mathcal{E},q,e_+,\hat \lambda, \eps)$, 
rescale $x\to x|\lambda|$, converting the adjoint eigenvalue ODE
$\tilde Z'=-G^* Z$ to
$$
\dot Z=B(\eps x)Z+\eps C(\eps x)Z,
$$
where $\eps:=|\lambda|^{-1}$, 
$\hat \lambda:=\frac{\lambda}{|\lambda|}$,
and $B=\bar{\hat \lambda} A^{-1,T}$, $C=-(EA^{-1})^T$, with $A$ and $E$ as in \eqref{ae}.

Noting that $A=\bp \alpha & 0\\ 0 & 0\ep -I$, where $\alpha$ is the
flux Jacobian for ideal gas dynamics, we find from standard gas-dynamical
computations that $A$ has distinct eigenvalues
$-1,-1, -1 \pm c$, where $c:=\frac{\sqrt{\Gamma (\Gamma +1)}}{\tau}$ 
denotes sound speed, and, for $0\le q<q_j$, these eigenvalues are
uniformly bounded away from zero as well as from each other.
By standard matrix perturbation theory \cite{K}, there thus exists
a smooth coordinate transformation $T=T(B)$ such that $T^{-1} B T$
is diagonal.
Making the change of coordinates $Z=TY$,
we thus obtain 
\be\label{Yeq}
\dot Y= T^{-1}(B+\eps C)TY -T^{-1}\dot T  Y
=: B_1Y + \eps C_1 Y,
\ee
where 
\be\label{B1}
B_1:=T^{-1}BT=
\bar {\hat \lambda} \bp -1+c & 0 & 0 & 0\\
0& -1-c & 0 & 0 \\
0& 0 & -1 &  0 \\
0& 0& 0 & -1  \\
\ep
\ee
and, 
since 
$\dot T=(\partial T /\partial B) \eps (\partial B/\partial x)
=O(\eps e^{-\theta |x|})$
and $|E-E(-\infty)|\le e^{-\theta | x|}$ for $ x\le 0$, $\theta>0$,
 \be\label{C1}
C_1:= T^{-1}CT -\eps^{-1}T^{-1}\dot T = 
( T^{-1}CT)(-\infty)+
O(e^{-\theta \eps |\hat x|}).
\ee

By block structure of $A$, we may take without loss
of generality  $T=\bp * & *\\ * &I_2\ep$, whence
$( T^{-1}CT)(-\infty)= \bp 0 & 0\\ 0 &E_{22}^T\ep$, where
$E_{22}=\bp qk\\-k\ep (d\phi(\bar T) \bar z  , \phi(\bar T))$
is rank one but (for large $\E$, in particular) not always diagonalizable.
Nonetheless, at $\hat x=-\infty$, we have
$$
E_{22}(-\infty)=
\bp qk\\-k\ep (0  , \phi(T_-)) =
\bp 0 & qk\phi(T_-)\\
0& -k\phi(T_-)\ep
$$
diagonalizable, so, by a further modification of $T_{22}$, we may
take without loss of generality
\be\label{finalE22}
( T^{-1}CT)_{22}(-\infty)=\bp 0 & 0\\ 0 & -k\phi(T_-)\ep
\ee
and $\eps( T^{-1}CT)_{22}=\eps \bp 0 & 0\\ 0 & -k\phi(T_-)\ep 
+O( \eps e^{-\theta |x|})$.
Converting briefly back to $x$-coordinates and applying Lemma \ref{conjlem},
we thus find that there is a nonsingular coordinate transformation
$W=SX$, $|S|,|S^{-1}|\le C$, converting
$\dot W= \eps( T^{-1}CT)_{22}W$ to the constant-coefficient equation
$\dot X=\eps \bp 0 & 0\\ 0 & -k\phi(T_-)\ep X$. 
Incorporating this further coordinate change in the $2$-$2$ block
(only), we obtain, finally 
\be\label{Xeq}
\dot X= B_2 X + \eps C_2 X,
\ee
\be\label{B2}
B_2=
\bar {\hat \lambda} \bp -1+c & 0 & 0 & 0\\
0& -1-c & 0 & 0 \\
0& 0 & -1 &  0 \\
0& 0& 0 & -1-\eps k\phi(T_-)/\bar {\hat \lambda}  \\
\ep,
\;
C_2=
\bp
O(e^{-\theta \eps |\hat x|}& O(e^{-\theta \eps |\hat x|}\\
O(e^{-\theta \eps |\hat x|}& 0_2
\ep,
\ee
where $O_2$ denotes the $2\times 2$ zero matrix.

\br
The reduction just performed, using the conjugation lemma to
diagonalize the lower righthand block, is a delicate
point of the analysis, avoiding potential difficulties associated with
turning points where $E_{22}$ becomes nondiagonalizable.\footnote{
Unimportant for $|\lambda|>>1+\mathcal{E}$, these dominate
behavior in the high-activation energy limit $\mathcal{E}\to \infty$
\cite{BZ,Z5}.}
\er

Using again standard matrix perturbation theory, it follows for
$\eps$ sufficiently small that there
is a further smooth coordinate transformation 
$$
Q=\bp
1&  O(\eps e^{-\theta \eps |\hat x|}) &  O(\eps e^{-\theta \eps |\hat x|})
&  O(\eps e^{-\theta \eps |\hat x|}) \\
O(\eps e^{-\theta \eps |\hat x|}) & 1& O(\eps e^{-\theta \eps |\hat x|}) 
&  O(\eps e^{-\theta \eps |\hat x|})\\
O(\eps e^{-\theta \eps |\hat x|}) &  O(\eps e^{-\theta \eps |\hat x|}) & 
1& 0 \\
O(\eps e^{-\theta \eps |\hat x|}) &  O(\eps e^{-\theta \eps |\hat x|}) & 
0 & 1\ep
$$
such that $M:=Q^{-1}(B_2+\eps C_2) Q$ is block-diagonal,
\be\label{M}
M=
\bp \bar{\hat \lambda}(-1+c) +O(\eps e^{-\theta \eps |\hat x|})& 0 & 0 &0 \\
0 & \bar{\hat \lambda}(-1-c) +O(\eps e^{-\theta \eps |\hat x|})& 0  &0\\
0 & 0 &
-\bar {\hat \lambda} & 0\\ 
0 & 0 & 0 &-\bar{\hat \lambda}-\eps k\phi(T_-)\ep,
\ee
and $\Theta:= - Q^{-1}\dot Q=O(\eps^2e^{-\theta \eps |\hat x|})$, 
taking the equations to the approximate block-diagonal form
treated in Lemmas \ref{reduction} and \ref{varconlem}, of
\be\label{finaldiag}
 W'= M W + \Theta W .
\ee

\subsection{High-frequency stability}\label{s:HFstab}

\bpr\label{hfstabprop}
For $k=1$, $\Gamma$, $\mathcal{E}$ bounded, and $q$ bounded away from $q_{cj}$,
ZND detonations are stable with respect to sufficiently
high frequencies; that is, there exists $R>0$ independent of 
$(\Gamma,\mathcal{E},q,e_+)$ 
such that $D_{ZND}(\lambda) \ne 0$ for $ \Re \lambda \ge 0$ 
and $ |\lambda|\ge R$.
\epr

\begin{proof}
{\it Case (i)} ($\Re \lambda << \eps |\lambda|$)
Equivalently, $\Re \hat \lambda << \eps $, whence, in \eqref{M},
there is a spectral gap between the fourth diagonal entry, 
$-\bar{\hat \lambda}-\eps k\phi(T_-)$, which has real
part $\le -\eps \eta$ for $\eta=\phi(T-)>0$, and the first three
diagonal entries, $\bar{\hat \lambda }(-1\pm c)$ and
$-\bar {\hat \lambda} $, which have real parts $\ge -C\Re \bar {\hat \lambda}
>> - \eps$.

Applying Lemma \ref{reduction} to \eqref{finaldiag}, we find that there
is a graph $W_4=\Phi(W_1,W_2,W_3)$, $\Phi= O(\eps)$ that is invariant
under \eqref{finaldiag}, from which we may reduce to an equation
 \be\label{checkeq}
\check W'= \check M \check W + \check \Theta \check W 
\ee
on $\check W=(W_1,W_3,W_3)$ alone, with
\be\label{checkM}
\check M=
\bp \bar{\hat \lambda}(-1+c) +O(\eps e^{-\theta \eps |\hat x|})& 0 & 0  \\
0 & \bar{\hat \lambda}(-1-c) +O(\eps e^{-\theta \eps |\hat x|})& 0  \\
0 & 0 & -\bar {\hat \lambda} \\ 
\ep
\ee
satisfying \eqref{neutral} with $\delta_*<<\eps$
and $\check \Theta= O(\eps^2e^{-\theta \eps |\hat x|})$.

Applying now
Lemma \ref{varconlem} to \eqref{checkeq}, we find that there
is a coordinate transformation $\check W=PX$, $P=I+O(\eps)$, taking
\eqref{checkeq} to the block-decoupled equation $\dot X=\check MX$,
of which the unique (up to constant multiplier)
solution $X^-$ decaying as $\hat x\to -\infty$
has value $X^-(0)$ at $\hat x=0$ parallel to $(1,0,0,)^T$.
Untangling coordinate changes, we find that 
unique (up to constant multiplier)
solution $W^-$ of \eqref{finaldiag} decaying as $\hat x\to -\infty$
is parallel to $I+O(\eps)$ times $(1,0,0,0)^T$ and
the unique (up to constant multiplier) solution
 $\tilde Z^-(0)$ of 
adjoint eigenvalue equation $\tilde Z'=-G^* Z$ is
parallel to $I+O(\eps)$ times the left unstable eigenvector of 
$A^{-1}(0)$, or $(\ell, 0)^T$, where $\ell$ is the left unstable
eigenvector of the gas-dynamical flux Jacobian $\alpha(0)$.

Likewise, $ \lambda[\bar {W}]+R(\bar W(0^-))$ is parallel to
$I+O(\eps)$ times $[\bar W]= ([u],0)^T$, whence, combining these
facts, we find that $\hat D_{ZND}(\lambda)=
\tilde Z^-(0)\cdot ( \lambda[\bar {W}]+R(\bar W(0^-)))$ is proportional
to $1+O(\eps)$ times $\Delta(\lambda):=\ell \cdot \lambda [u]$, which
may be recognized as the Lopatinski determinant for an ideal gas-dynamical
shock, known by \cite{Er4,M} to be nonvanishing on
$\Re \lambda \ge 0$ except at $\lambda=0$, with all constants uniform
in model parameters and $\Re \lambda \ge 0$.
For $|\lambda|$ sufficiently large, therefore, or equivalently,
$\eps:=|\lambda|^{-1}$ sufficiently small, we find that 
$\hat D_{ZND}(\lambda)\ne 0$.

{\it Case (ii)} ($\Re \lambda \ge C^{-1} \eps |\lambda|$)
Equivalently, $\Re \hat \lambda \ge C^{-1}\eps $, whence, applying
Lemma \ref{reduction} to \eqref{finaldiag}, we find that there
is a graph $(W_2,W_3,W_4)=\Phi(W_1)$, $\Phi= O(\eps)$ that is invariant
under \eqref{finaldiag}, from which we find that the unique
unique (up to constant multiplier) solution $W^-$ of \eqref{finaldiag}
decaying as $\hat x\to -\infty$
has value $W^-(0)$ at $\hat x=0$ parallel to $I+O(\eps)$
times $(1,0,0,0)^*$.
Untangling coordinate changes, and arguing as in the previous case,
we thus find again that $\hat D_{ZND}(\lambda)=
\tilde Z^-(0)\cdot ( \lambda[\bar {W}]+R(\bar W(0^-)))$ is proportional
to $1+O(\eps)$ times the gas-dynamical Lopatinski determinant
$\Delta(\lambda):=\ell \cdot \lambda [u]$,
$\ell \cdot [u] \ne 0$, 
hence nonvanishing for $|\lambda|$ sufficiently large.
\end{proof}

\br\label{mult}
\textup{
Applying Lemmas \ref{reduction} and \ref{varconlem} in 
sequence in this way, 
one may treat the situation arising in the multi-dimensional case 
(see \cite{Er3}) of an approximately block-diagonal system for 
which some blocks have a uniform spectral gap and others have uniformly 
small spectral gap.
We hope to report on this in future work.
}
\er

\br\label{fullconj}
\textup{
Though we did state it, the arguments above show that there
exists a change of coordinates $Q=I+O(\eps e^{-\theta \eps |\hat x|})$
taking \eqref{finaldiag} to exactly diagonal form
$W'=MW$, where $M$ is as in \eqref{M}.\footnote{
This follows by separating off scalar diagonal entries with spectral
gap from other entries using Lemma \ref{reduction}, to obtain
scalar equations $w'=mw+\theta w$ with 
$|\theta|=O(\eps^2 e^{-\theta \eps|\hat x|})$ for which error $\theta$
can be shown to be negligible by explicit exponentiation.
(Conjugation of the nonscalar block \eqref{checkeq} has already been shown.)}
This gives information about the full flow, and not only the decaying
solution important for the stability theory.
Note that, in the exactly diagonal coordinates $W$, the first
two entries correspond to coefficients of the first two eigenvectors
of $A$ in the eigenexpansion of $\tilde Z$, while the second two
entries correspond to unknown linear combinations of the the
coefficients of the third and fourth eigenvectors of $A$.\footnote{
Recall, these depend on the abstract conjugation prescribed in
going from \eqref{B1} to \eqref{B2}.}
}

\textup{
Likewise, a closer look at the proof reveals the asymptotic description
\be\label{asymptotics}
D_{ZND}(\lambda)=  e^{C_1 \lambda +C_0}\Delta(\lambda) (1+O(\eps)),
\ee 
where $\Delta(\lambda):=\ell \cdot \lambda [u]$, $\ell \cdot [u]\ne 0$
is the Lopatinski determinant associated with the Neumann
shock.
This can be used as in \cite{HLyZ1} as the basis of a convergence study,
to obtain practical bounds on unstable eigenvalues.
Higher order approximants
$$
D_{ZND}(\lambda)=  e^{C_1 \lambda +C_0 + C_{-1}\lambda^{-1}}\Delta(\lambda) (1+D_1\eps + O(\eps^2)),
$$
etc., may be obtained by further diagonalizations as detailed in \cite{MaZ3}.
}
\er

\bc\label{main}
Under the parametrization \eqref{param},
ZND detonations are stable in both the small-heat release and
(with Erpenbeck's scaling, Section \ref{s:f}) high-overdrive limits.
\ec

\medskip
{\bf Acknowledgement.}
Thanks to the University of Paris 13 for their hospitality
during a visit in which this work was partly carried out.
Thanks to Olivier Lafitte, Benjamin Texier, and Mark Williams
for stimulating discussions
regarding stability of ZND detonations.



\appendix
\section{Numerical implementation}
For purpose of numerical applications, we provide also
a simpler high-frequency argument requiring the weaker estimate
\be\label{replaced}
\Phi=\bp
O(\eps e^{-\theta \eps |\hat x|}) & O(\eps^2 e^{-\theta \eps^2 |\hat x|}) & 
 O(\eps^2 e^{-\theta \eps |\hat x|}) & O(\eps^2 e^{-\theta \eps^2 |\hat x|}) \\
O(\eps^2 e^{-\theta \eps |\hat x|}) & O(\eps e^{-\theta \eps |\hat x|}) & 
 O(\eps e^{-\theta \eps |\hat x|}) & O(\eps e^{-\theta \eps |\hat x|}) \\
O(\eps^2 e^{-\theta \eps |\hat x|}) & O(\eps e^{-\theta \eps |\hat x|}) & 
 O(\eps e^{-\theta \eps |\hat x|}) & O(\eps e^{-\theta \eps |\hat x|}) \\
O(\eps^2 e^{-\theta \eps |\hat x|}) & O(\eps e^{-\theta \eps |\hat x|}) & 
 O(\eps e^{-\theta \eps |\hat x|}) & O(\eps e^{-\theta \eps |\hat x|}) \\
\ep
\ee
rather than $\Theta= O(\eps^2 e^{-\theta \eps |\hat x|})$
in \eqref{finaldiag}, removing the need for the intermediate
coordinate transformation $S$ in the lower righthand $2\times 2$ block.
This avoids an abstract conjugation step that is difficult to
estimate efficiently numerically, and also provides a slightly
simpler proof treating all frequencies at once instead of dividing
into cases.
On the other hand, it provides information only about the single
mode $\tilde Z^-$ associated with the stability determinant, and
not the entire flow of the adjoint eigenvalue ODE, which may be
of interest in more general situations.

\subsection{Variable-coefficient gap lemma}\label{s:vargap} 

Consider a first-order system \eqref{newsys} satisfying
\eqref{udecay2} and
\be\label{neutral2}
\Re M^p(x,\eps)\le \eps \delta^p(\eps) + C \eps e^{-\theta \eps |x|}
\ee
for some uniform $C,\theta>0$, all $x\le 0$, 
$\Re M:=\frac{1}{2}(M+M^*)$.
(Note, in contrast with \eqref{neutral}, 
that this is a bound from above only.)
Assume, further, that there exists a smooth vector $V_*^p(x)$,
$1/C\le|V_*^p|\le C$,
for which $M^pV^p_*\equiv 0$.

\bl \label{vargaplem}
Assuming \eqref{neutral2},
for $\delta^p(\eps)\le \delta_*$ sufficiently small, and $\eps>0$ 
sufficiently small, there exists a solution
$V^p(x,\eps)$ of \eqref{newsys} defined on $x\le 0$ such that
\begin{equation}
\label{Pdecay2new} 
|( V^p-V^p_*)(x)|  \le C_1 \eps e^{- \theta \eps |x|/2}|V_*^p|
\quad
\text{\rm for } x\le 0.
\end{equation}
\el

\begin{proof}
We seek, equivalently, a 
solution $V^p$ of the integral fixed-point equation
\begin{equation}
\begin{aligned}
\CalT V(x) 
&= V^p_*(x)+  \int^x_{-\infty} \mathcal{F}^{y\to x} \Theta (y)V(y) dy ,
\end{aligned}
\end{equation}
where $\mathcal{F}^{y\to x}$ is the solution operator of
$V'=M^p V$ from $y$ to $x$.
From \eqref{neutral2}, we obtain by an easy energy estimate like
that of \eqref{enest} the bound
\be\label{growth2}
\|\mathcal{F}^{y\to x}\|\le Ce^{  \eps \delta^p(\eps) (x-y)}
\le Ce^{  \eps \delta_* (x-y)}
\; \hbox{ \rm for } \; x>y.
\ee 

For $ \delta_*\le \theta/2$ and $\eps>0$ sufficiently small, this implies that
$\mathcal{T}$ is a contraction on $L^\infty(-\infty, 0]$.
For, applying \eqref{udecay2} and \eqref{growth2}, we have
\begin{equation}\label{con2}
\begin{aligned}
\left|\CalT V_1 - \CalT V_2 \right|_{(x)} 
&\le C\eps^2 |V_1 - V_2|_\infty 
\int^x_{-\infty} e^{\theta \eps (x-y)/2} e^{\theta \eps y} dy 
\le  C_1 \eps |V_1 - V_2|_\infty e^{-\theta \eps |x|/2},
\end{aligned}
\end{equation}
which for $\eps$ sufficiently small is less than $\frac{1}{2} |V_1 - V_2|_\infty$.
By iteration, we thus obtain a solution 
$V \in L^\infty (-\infty, 0]$ of $V = \CalT V$. 
Further, taking $V_1=V$, $V_2=0$ in \eqref{con}, we obtain,
using contraction together with 
the final inequality in \eqref{con2}, that 
$ 
|V- V_*|_{L^\infty(-\infty,x)}=
|V- \CalT(0)|_{L^\infty(-\infty,x)} \le  
\frac{1}{2}|V-0|_{L^\infty(-\infty,x)}, $
 yielding \eqref{Pdecay2new} as claimed.
\end{proof}

\subsection{Alternate high-frequency analysis}

\begin{proof}[Alternate proof of Prop. \ref{hfstabprop}]
By inspection, $\M:=(M+\check \Theta)- (M+\check\Theta)_{11}I$
satisfies \eqref{neutral2} for $M$ as in \eqref{M},
$\eps:=|\lambda|^{-1}$,
and 
$$
\check \Theta=\bp
\Theta_{11}& 0&0&0\\
0& \Theta_{22}& \Theta_{23}& \Theta_{24}\\
0& \Theta_{32}& \Theta_{33}& \Theta_{34}\\
0& \Theta_{42}& \Theta_{43}& \Theta_{44}\\
\ep
$$
with $\Theta$ as in \eqref{replaced}, 
with $\M V_*\equiv 0$ for $V_*:=(1,0,0,0)^T$, whence, applying Lemma
\ref{vargaplem}, we obtain a decaying solution 
$W(x;\eps)=\omega(x)e^{\bar{\hat \lambda}(-1+c)\hat x}V(x;\eps)$ of
\eqref{finaldiag} converging as $O(\eps)$ in relative error
to $(1,0,0,0)^T$, where 
$\omega(\hat x)=e^{\int_{-\infty}^{\hat x} O(\eps e^{-\theta \eps
|\hat y|}d\hat y}>0$ is uniformly bounded above and below.
Untangling coordinate changes, and 
noting that $ \lambda[\bar {W}]+R(\bar W(0^-))$ is parallel to
$I+O(\eps)$ times $[\bar W]= ([u],0)^T$, 
we thus find that $\hat D_{ZND}(\lambda)=
\tilde Z^-(0)\cdot ( \lambda[\bar {W}]+R(\bar W(0^-)))$ is proportional
to $1+O(\eps)$ times the gas-dynamical Lopatinski determinant
$\Delta(\lambda):=\ell \cdot \lambda [u]$,
$\ell \cdot [u] \ne 0$, 
hence nonvanishing for $|\lambda|$ sufficiently large.
\end{proof}

\end{document}